\documentclass{amsart}

\usepackage{amsmath,amsthm,amsfonts,amssymb,latexsym}

\newtheorem{thm}{Theorem}[section]
\newtheorem{cor}[thm]{Corollary}
\newtheorem{lemma}[thm]{Lemma}
\newtheorem{remark}[thm]{Remark}

\newtheorem{open}[thm]{Open Problem}
\theoremstyle{definition}
\newtheorem{defn}[thm]{Definition}

\newcommand{\from}{\colon}

\newcommand{\Z}{\mathbb{Z}}
\newcommand{\N}{\mathbb{N}}
\newcommand{\R}{\mathbb{R}}
\newcommand{\F}{\mathbb{F}}

\renewcommand{\subset}{\subseteq}
\renewcommand{\supset}{\supseteq}

\newcommand{\card}[1]{\lvert #1 \rvert}

\newcommand{\union}{\cup}
\newcommand{\bigunion}{\bigcup}
\newcommand{\inters}{\cap}

\DeclareMathOperator{\Stab}{Stab}
\DeclareMathOperator{\Irr}{Irr}

\newcommand{\<}{\langle}
\renewcommand{\>}{\rangle}

\newcommand{\define}[1]{\textbf{#1}} 

\newcommand{\actson}{\curvearrowright}

\newcommand{\powset}{\mathcal{P}}
\newcommand{\PP}{\powset_{\text{pr}}}

\newcommand{\SO}{\mathsf{SO}}

\newcommand{\PSL}{\mathsf{PSL}}
\newcommand{\PGL}{\mathsf{PGL}}
\renewcommand{\P}{\mathsf{P}}

\newcommand{\E}{\mathrel{E}}
\newcommand{\comp}[1]{{#1}^{\text{c}}}

\begin{document}

\title{Measurable realizations of abstract systems of congruences}

\thanks{The first author was partially supported by NSF grant DMS-1500906 and
DMS-1855579. The second author was partially supported by NSF grant
DMS-1500974 and DMS-1764174. The third author was partially supported by NSF grant
DMS-1700425}

\author{Clinton T.~Conley}
\address{Department of Mathematical Sciences,
Carnegie Mellon University, Pittsburgh, PA 15213}
\email{clintonc@andrew.cmu.edu}

\author{Andrew S. Marks}
\address{Department of Mathematics, University of California at Los Angeles}
\email{marks@math.ucla.edu}

\author{Spencer T. Unger}
\address{Department of Mathematics, Hebrew University of Jerusalem}
\email{unger.spencer@mail.huji.ac.il}

\date{\today}

\begin{abstract}
An abstract system of congruences describes a way of partitioning a space
into finitely many pieces satisfying certain congruence relations. Examples
of abstract systems of congruences include paradoxical decompositions and
$n$-divisibility of actions. We consider the general question of when there
are realizations of abstract systems of congruences satisfying various
measurability constraints. We completely characterize which abstract
systems of congruences can be realized by nonmeager Baire measurable pieces
of the sphere under the action of rotations on the $2$-sphere. This answers
a question of Wagon. We also construct Borel realizations of abstract
systems of congruences for the action of $\PSL_2(\Z)$ on $\P^1(\R)$. The
combinatorial underpinnings of our proof are certain types of decomposition
of Borel graphs into paths. We also use these decompositions to obtain some
results about measurable unfriendly colorings.
\end{abstract}

\maketitle

\section{Introduction}

Recently, several results have been proved about the extent to which realizations of
geometrical paradoxes can be found with sets having measurability
properties such as being Borel, Lebesgue measurable, or Baire measurable (see for
instance \cite{CS}\cite{DF}\cite{GMP16}\cite{GMP17}\cite{Ma}\cite{MU16}\cite{MU17}). This is
a growing area of study at the interface of descriptive set theory,
combinatorics, and ergodic theory. This paper is a contribution to this study.
One of the
earliest results in this vein is the theorem of Dougherty and Foreman \cite{DF}
that the Banach-Tarski paradox can be realized using Baire measurable
pieces. In contrast to the classical Banach-Tarski paradox which uses five
pieces, Dougherty and Foreman's Baire measurable solution uses six pieces.
A result of Wehrung \cite{Weh} implies that this is optimal; there is no
Baire measurable realization of the Banach-Tarski paradox with five pieces.
This suggests a subtle difference
between the classical and Baire measurable contexts.

In this paper, we consider a refined framework called ``abstract systems of
congruence'' for describing when an
action can be partitioned into finitely many pieces satisfying certain
congruence relations. As one application, we give an exact
characterization of which abstract systems of congruences can be realized in
the $2$-sphere with arbitrary pieces versus nonmeager Baire measurable
pieces. This refines the dual results of Wehrung \cite{Weh}, and Dougherty
and
Foreman \cite{DF}.

We formally define abstract systems of congruences as follows.
Given a set $S$, its \define{proper powerset} $\PP(S)$ is $\PP(S) = \{R \subset S
\colon R \neq \emptyset \land R \neq S\}$.
Following Wagon \cite[Definition 4.10]{W}, an \define{abstract system of
congruences} on $n  = \{0, \ldots, n-1\}$ is an equivalence relation
$E$ on $\PP(n)$, so that if $U \E V$, then $\comp{U} \E
\comp{V}$. Here $\comp{U}$ denotes the complement of $U$.
Suppose $a \from \Gamma \actson X$ is an
action of a group on a set $X$. Then we say that $A, B \subset X$ are
$a$-congruent if there is a group element $\gamma \in \Gamma$ such that
$\gamma \cdot A = B$. An \define{$a$-realization} of an abstract
system of congruences $E$ is a
partition $\{A_0, \ldots, A_{n-1}\}$ of $X$ such that for all $U,V \in
\PP(n)$ with $U \mathrel{E} V$, we have that $\bigunion_{i \in U} A_i$ and
$\bigunion_{i \in V} A_i$ are $a$-congruent. The definition of an
abstract system of congruences reflects that fact that congruence is an
equivalence relation, and that if $A, B \subset X$ are congruent, then
$\comp{A}$ and $\comp{B}$ are also congruent.

An important example of an abstract system of congruences is the smallest
abstract system of congruences $E$ on
$\PP(4)$ containing the relations $\{0\} \E \{0, 1, 2\}$ and $\{1\} \E
\{0, 1, 3\}$. A realization of this system gives a paradoxical
decomposition, since $\{0,3\}$ and $\{1,2\}$ partition $\{0,1,2,3\}$.
The
translation action of the free group on two generators $\F_2$ on itself is an
example of an action realizing this system of congruences \cite[Theorem
4.2]{W}.
Another important
example of an abstract system of congruences is the smallest abstract
system of congruences
$E$ on $\PP(n)$ where $\{i\} \E \{j\}$ for every $i,j \in n$. An action is said
to be \define{$n$-divisible} if it satisfies this system of congruences (that it, is can
be partitioned into $n$ congruent pieces).  For example, it is easy to see that
the action of the rotation group $\SO_{3}$ on the $2$-sphere is not
$2$-divisible by considering the ``poles'' of the rotation.  However, this
action is $n$-divisible for $n \geq 3$ (see \cite[Corollary 4.14]{W}).

Wagon has characterized which abstract systems of congruences can be
realized in the action of the group $\SO_3$ of rotations on the $2$-sphere.
We say that an abstract
system of congruences $E$ on $n$ is
\define{non-complementing} if there is no set $X \in \PP(n)$
such that $X \E \comp{X}$.

\begin{thm}[{\cite[Corollary 4.12]{W}}]\label{Wagon_abstract} Suppose $E$ is an
abstract system of congruences.
$E$ can be realized in the action of $\SO_3$ on the $2$-sphere if and only if $E$ is non-complementing.
\end{thm}

We show that in order to realize an abstract system of congruences with Baire
measurable pieces in the sphere, we need one additional property. Say that an
abstract system of congruences $E$ on $\PP(n)$ is \define{non-expanding} if
there do not exist sequences of sets $(V_i)_{i \leq k}$ and $(W_i)_{i \leq k}$
where $V_i \E W_i$ for every $i \leq k$ and $W_i \subset V_{i+1}$ for every $i <
k$, such that $W_k \subsetneq V_0$. Hence,
\[V_0 \E W_0 \subset V_1 \E W_1 \subset \ldots V_k \E W_k \subsetneq V_0.\]

\begin{thm}\label{Baire_abstract}
  Suppose $E$ is an abstract system of congruences.
  Then $E$ can be realized in action of $\SO_3$ on the
  $2$-sphere with Baire measurable pieces each of which is nonmeager if
  and only if $E$ is non-complementing and non-expanding.
\end{thm}

This theorem positively answers Wagon's question \cite[Page 47]{W} of
whether the $2$-sphere is $n$-divisible with Baire measurable pieces for $n
\geq 3$. Indeed, the smallest abstract system of congruences $E$ containing the
relations $\{1\} \mathrel{E} \{2\} \mathrel{E} \ldots \mathrel{E} \{n\}$ is
clearly noncomplementing and nonexpanding for $n \geq 3$, and hence has a
Baire measurable realization in the action of $\SO_3$ on the $2$-sphere.
Wagon has also asked whether the $2$-sphere is $n$-divisible into
Lebesgue measurable pieces. This remains an open problem.

Let $\F_n$ be the free group on $n$ generators.
Our proof of Theorem~\ref{Baire_abstract} shows more generally that if $n
\geq 2$, then
any free Borel action of $\F_n$
on a Polish space $X$ can realize an abstract system of congruences that is
non-expanding and non-complementing using Baire measurable pieces. (See
Lemma~\ref{realization_lemma}).

Our main tool for proving Theorem~\ref{Baire_abstract} is a decomposition
lemma for acyclic Borel graphs into sets of paths with a property
concerning how the paths from different sets may overlap.

\begin{defn}
  Suppose $G$ is a graph and $G_0, G_1, \ldots$ is a sequence of subgraphs
  of $G$. Then we say $G_0, G_1, \ldots$ is \define{end-ordered} if for
  all vertices $x$ in $G$, if $x$ is a vertex in $G_i$ and $G_j$ where $i <
  j$, then $x$ is a leaf in $G_j$. Similarly, if $S_0, S_1, \ldots$ are
  sets of subgraphs of $G$, then we say that $S_0, S_1, \ldots$ is
  \define{end-ordered} if for all vertices $x$ in $G$, if $x$ is vertex in
  $H \in S_i$ and a vertex in $H' \in S_j$ where $i < j$, then $x$ is a
  leaf in $H'$.
\end{defn}

\begin{defn}
  Suppose $G$ is an acyclic Borel graph. Then a \define{path decomposition} of
  $G$ is a sequence $P_0, P_1, \ldots$ of sets of paths in $G$ such
  that $P_0, P_1, \ldots$ is
  end ordered, every
  $P_i$ consists of vertex disjoint paths, and for every edge $e$ in $G$,
  there exists exactly one $P_i$ so that $e$ appears in a path in $P_i$.
  We
  say that a path decomposition is Borel if each set $P_i$ is Borel, and
  the path decomposition has length at least $n$ if every path has length
  at least $n$.
\end{defn}

Roughly, a path decomposition is a way of covering the graph with sets of
paths $P_0, P_1, \ldots$ so that all the paths in $P_j$ have interiors that
are disjoint from the paths in $P_i$, for $i < j$.

One of our main lemmas (Lemma~\ref{path_decomp}) says that if $G$ is a
locally finite acyclic Borel graph, then for all $n$, there is a comeager set on
which $G$ has a Borel path decomposition of length at least $n$.

A different case in which we have Borel path decompositions is when we have
Borel end selections.
Recall that if $G$ is a graph on $X$, a \define{ray} is an
infinite simple path in $G$, and that two rays $(x_i)_{i \in \N}$ and $(y_i)_{i
\in \N}$ are \define{end-equivalent} if for every finite set $S \subset X$,
the rays $(x_i)$ and $(y_i)$ eventually lie in the same connected component
of $G \restriction (X \setminus S)$. An \define{end} of $G$ is an end-equivalence
class of $G$. If $G$ is a Borel graph on $X$, we say that \define{$G$
admits a
Borel selection of finitely $k$ ends in each connected component} if there
are Borel functions $r_0, \ldots, r_{k-1}$  sending each $x \in X$ to $k$
end-inequivalent rays $r_0(x), \ldots, r_k(x)$ in the connected
component of $x$ such that if $y$ are in the same connected component of
$G$ as $x$, then
$\{r_0(x), \ldots, r_{k-1}(x)\}$ and $\{r_0(y), \ldots, r_{k-1}(y)\}$
are representatives of the same set of ends.
We say that \define{$G$ admits a
Borel selection of finitely many ends in each connected component} if $G$
can be partitioned into countably many invariant Borel sets $A_0, A_1,
\ldots$
so that for each $i$, there is some $k$ so that $G \restriction A_i$
has a Borel selection of $k$ ends in each connected component.

We show that
if $G$ is an acyclic bounded degree Borel graph on $X$ such that there is
a Borel selection of finitely many ends in every connected component of
$G$, then for every $n$ we can find a Borel path decompositions of $G$ of
length at least $n$ (see
Lemma~\ref{end_sel_path_decomp}).
We construct explicit realizations of
abstract systems of congruences for the action of $\PSL_2(\Z)$
on $P^1(\R)$, by combining this lemma with an explicit end selection
defined using continued fraction expansions.

\begin{thm}\label{PSL2_thm}
  Suppose $E$ is an abstract system of congruences which is
  non-complementing and non-expanding. Then $E$ can be
  realized in the action of
  $\PSL_2(\Z)$ on $P^1(\R)$ by Borel pieces.
\end{thm}

For example, this action is $n$-divisible using Borel pieces for $n
\geq 3$.

\begin{open}
Characterize the abstract systems of congruences
which can be realized in the action of $\PSL_2(\Z)$ on $P^1(\R)$ by Borel pieces.
\end{open}

It is a theorem of Adams \cite[Lemma 3.21]{JKL} that if $G$ is a locally
finite graph on a standard probability space $(X,\mu)$ and $G$ is
$\mu$-hyperfinite, then $G$ admits a $\mu$-measurable selection of finitely
many ends. Using Adams's theorem, we also show that
any $\mu$-hyperfinite action of
$\F_2$ on a standard probability space $(X,\mu)$ has a $\mu$-measurable
realization of any abstract system of congruences $E$, if $E$ is
non-complementing and non-expanding. (See Theorem~\ref{mu_hyp_decomp}).

Our decomposition lemmas have some other applications in Borel
combinatorics. Simon Thomas has asked whether every locally
finite Borel graph has an unfriendly
Borel coloring, where an \define{unfriendly coloring} of a graph $G$ on $X$ is a function $f
\from X \to 2$ such that for every $x$,
\[\card{\{y \in N(x) \colon c(x) \neq c(y)\}} \geq \card{\{y \in N(x) \colon
c(x) = c(y)\}}.\]
Thomas's question is partially motivated by the open problem in classical
combinatorics of whether every countable graph admits an unfriendly
coloring. If $G$ is a graph on $X$, say that a function $f \from X \to 2$
is \define{strongly unfriendly} if for every $x$, $|\{y \in N(x) \colon
c(x) = c(y) \} \leq 1$.

We use our decomposition lemma to prove the following result:

\begin{thm}\label{unfriendly_coloring}
  Suppose $G$ is a locally finite acyclic Borel graph on a Polish space
  $X$ that admits a Borel path decomposition of length at least $5$. Then
  $G$ has a Borel strongly unfriendly coloring. Hence, if $G$ is a locally finite acyclic Borel
  graphs of degree at least $2$, then $G$ admits a Baire measurable
  strongly unfriendly colorings, and $G$ admits a $\mu$-measurable
  strongly unfriendly colorings for every Borel probability measure on $X$ rendering
  $G$ $\mu$-hyperfinite.
\end{thm}

In Section~\ref{comb_applications} we also discuss some further applications of
our decompositions, such as new proofs of Baire measurable and $\mu$-measurable
edge-coloring and matchings.

\subsection*{Acknowledgments}
The authors would like to thank Anton
Bernshteyn, Benson Farb, and Alekos Kechris for
helpful discussions about the material of the paper. The authors would also like
to thank the anonymous referee for their excellent suggestions and thorough reading of
the paper.

\section{Preliminaries}

Our notation for graph theory is standard, see \cite{D}. We recall a
few notions. By a graph on $X$ we mean a simple undirected graph with vertex set $X$.
The \define{degree} of a vertex is its number of neighbors. Two vertices
are \define{adjacent} if there is an edge between them. A vertex is a
\define{leaf} if it has degree $1$, and is a \define{splitting vertex} if
it has degree at least $3$. By a \define{path} we mean a simple path of
finite length $x_0, \ldots, x_n$. The endpoints of the path are $x_0$ and
$x_n$, and the remaining vertices are interior vertices of the path.
By a \define{ray}, we mean a simple infinite path
$(x_i)_{i \in \N}$.

If $G$ is a graph, we say a set of vertices is \define{independent} if it does not
contain two adjacent vertices. We say that a set $A$ is
\define{$k$-independent} if
for all distinct $x, y \in A$, we have $d(x,y) > k$.

Suppose $G$ is a graph on $X$. A \define{subgraph} $H$ of $G$ is a graph on
a subset of $X$ so that every edge in $H$ is an edge in $G$. If $Y \subset
X$, the \define{restriction of $G$ to $Y$} or \define{induced subgraph on
$Y$}, denoted $G \restriction Y$, is the graph on $Y$ where the edges of $G
\restriction Y$ are all edges in $G$ with vertices in $Y$.

A Borel graph is a graph on a Polish space $X$ whose edge relation is
Borel. For background on Borel graphs see \cite{KM}. An important example
of a Borel graph arises from Borel group actions. If $a$ is a Borel action of a
countable group $\Gamma$ on a Polish space $X$ and $S \subset \Gamma$ is a
symmetric set of group elements, then we let $G(a,S)$ be the graph on $X$
where $x, y$ are adjacent if there exists some $\gamma \in S$ such that
$\gamma \cdot x = y$.

If $G$ is a Borel
graph on $X$, the set of all paths of $G$ is a Borel subset of $\bigunion_n
X^n$, and hence a standard Borel space. Hence we may speak about a set of
paths in $G$ being Borel.

We note that in contrast to Lemmas~\ref{path_decomp} and
\ref{end_sel_path_decomp} there exist Borel graphs which do not admit Borel path
decompositions of length at least $3$.

\begin{thm}
  Suppose that $G$ is Borel graph of degree at least $3$ on a Polish space
  $X$ that admits an invariant measure $\mu$. Then $G$ does not admit a
  Borel path decomposition on any $\mu$-conull Borel set.
\end{thm}
\begin{proof}
  Let $P_0, P_1, \ldots$ be a Borel path decomposition. Note that every
  vertex $x$ must be the endpoint vertex of some unique path $p(x) \in P_i$ since $G$ has
  degree at least $3$. Define a Borel function $f \from X \to X$ where
  $f(x)$ is the vertex adjacent to $x$ in $p(x)$.
   Then $f$ is a compression function contradicting $\mu$ being measure
  preserving. \cite{N}
\end{proof}

We will use the following lemma giving a criterion for
the existence of abstract systems of congruences without any measurability
properties.

\begin{defn}\label{defn:good_generating}
  Suppose $E$ is an abstract system of congruences.
  Say that a relation $R$ on a set $X$ \define{generates} the
  equivalence relation $E$ on $X$ if the smallest abstract system of
  congruences containing $R$ is equal to $E$. Say that a generating set $R$
  for $E$ is \define{good} if $R$ contains all pairs $(U,V) \in E$ such
  that $U = \comp{V}$. Finally, a minimal good generating set of $E$ is a
  good generating set $R$ so that there is no proper subset of $R$ that is
  a good generating set for $E$.
\end{defn}

\begin{lemma}[See also {\cite[Section 4]{W}}]\label{wagon_realization_lemma}
  Suppose that $E$ is an abstract system of congruences on $n$, and $R =
  \{(S_1, T_1), \ldots,
  (S_k, T_k)\}$ is a minimal good generating set of $E$. Suppose $a \colon \Gamma \actson X$
  is an action of
  \[\Gamma = \langle \gamma_1 \ldots \gamma_k \mid
  \{\gamma_i^2 = 1 \colon T_i = \comp{S_i}\}\rangle. \]
  Suppose finally that for every $x \in X$, $\Stab(\{x\})$ is cyclic.
  Then there is an $a$-realization
  $\{A_0, \ldots, A_{n-1}\}$ of $E$, witnessed by
  \[\gamma_i \cdot \bigunion_{j \in S_i} A_j = \bigunion_{j \in
  T_i} A_j. \tag{*}\]
\end{lemma}
\begin{proof}
  This Lemma is proved in {\cite[Section 4]{W}} when $E$ is
  non-complementing.

  Using the axiom of choice, it suffices to prove the lemma when the action
  has the single orbit.
  Since the stabilizer of every point is cyclic, the graph $G(a, \{\gamma_i : i \leq k\})$ has at most one cycle.

  Suppose there is a cycle $x_0, x_1, \ldots, x_l = x_0$. Let $g$ be
  the group element $g = g_{l-1} \ldots g_1 g_0$ so that $x_{i+1} = g_i \ldots g_1
  g_0 \cdot x_0$, and $g_i \in \{\gamma_1^{\pm}, \ldots,
  \gamma_k^{\pm}\}$. We claim we can assign elements of this cycle to $A_0, \ldots,
  A_{n-1}$ in a way that is consistent with (*). First, define functions
  $X$ and $Y$ on the generators by letting $X(\gamma_j) = S_j$,
  $Y(\gamma_j) = T_j$, $X(\gamma^{-1}) = T_j$, and $Y(\gamma^{-1}) = S_j$,
  so obeying (*) corresponds to having
  \[x_i \in \bigunion_{j \in X(g_i)} A_j \text{ iff } x_{i+1} \in
  \bigunion_{j \in Y(g_i)} A_j.\]
  Let $i^+$ denote $i + 1 \bmod k$.

  Case 1: Suppose there is some $i< l$ such that $X(g_{i^+}) \neq Y(g_{i})$
  and $X(g_{i^+}) \neq \comp{Y(g_{i})}$. Then we claim we can assign $x_0, \ldots, x_l$ to
  $A_0, \ldots, A_{n-1}$ in a way that satisfies (*). For example, suppose
  there is $r, s \in X(g_{i^+})$ such that $r \in Y(g_i)$ and $s
  \in \comp{Y(g_i)}$. Start by assigning $x_{i^{++}}$ to an
  arbitrary element of $Y(g_{i^+})$. Then proceed around the
  cycle, assigning elements in a way consistent with (*). Finish by
  assigning $x_{i^+}$ to $A_{r}$ if $x_i \in X(g_i)$, or assigning
  $x_{i^+}$ to $A_{s}$ if $x_i \notin X(g_i)$. The other cases are
  essentially identical.

  Case 2: Suppose for all $i< l$, $X(g_{i^+}) = Y(g_{i})$ or $X(g_{i^+}) =
  \comp{Y(g_{i})}$. In this case, we claim that if there is no way to
  assign $x_0, \ldots, x_l$ to
  $A_0, \ldots, A_{n-1}$ in a way that satisfies (*), then $R$ is not a
  minimal good generating set, which is a contradiction.

  Let $V(0) = X(g_0)$, and then inductively define $V(i+1) = Y(g_i)$ if
  $V(i) = X(g_i)$, and otherwise $V(i+1) = \comp{V(g_i)}$ if $V(i) =
  \comp{X(g_i)}$. Hence $V(0) \mathrel{E} V(1) \mathrel{E} V(2) \ldots
  \mathrel{E} V(l)$. Since there is no way to assign $x_0, \ldots, x_l$ to
  $A_0, \ldots, A_{n-1}$ in a way that satisfies (*), we must have that
  $V(0) = \comp{V(l)}$.
  Now take a minimal length subsequence $V(i), \ldots, V(j)$ of $V(0),
  \ldots, V(l)$ such that
  \[\text{$j - i \geq 2$ and $V(i) = V(j)$ or $V(i) = \comp{V(j)}$.} \tag{**}\]
  It is clear that if
  $g_i = \gamma_m$, then we can remove the pair
  $(S_m, T_m)$ from $R$ and we would still generate $E$. This is because if $V(i) = V(j)$,
  then $V(i) \E V(i+1)$ follows from $V(i+1) \E V(i + 2) \ldots \E V(j) =
  V(i)$. If $V(i) = \comp{V(j)}$, then the fact that $V(i) \E V(i+1)$
  follows from $V(i+1) \E V(i + 2) \ldots \E V(j) = \comp{V(i)}$, and since
  by the definition of a good generating set, the pair $V(j) \E
  \comp{V(i)}$ must appear in $R$.
  (Note that here we are using the minimal length of this subsequence among
  those with (**) and the fact that $g$ is a reduced word to ensure that the equivalences $V(i) \E V(i+1)$ and
  $V(i) \E \comp{V(i+1)}$ do not appear in the equivalences $V(i+1) \E \ldots
  \E V(j)$). This finishes Case 2.

  Now that we assigned the elements of the cycle to $A_0, \ldots, A_{n-1}$,
  if a cycle exists, for the remaining acyclic portion of the graph, we
  clearly iteratively assign
  the vertices to $A_0, \ldots, A_{n-1}$ in a way that satisfies (*).
\end{proof}

Throughout we will be working with actions of such groups $\Gamma$ that are
free products of copies of $\Z$ and $\Z/2\Z$, and where the generators of
$\Gamma$ of order $2$ will witnesses congruences of the form $U \mathrel{E} \comp{U}$.

\section{Baire measurable realizations}

In this section we prove Theorem~\ref{Baire_abstract}. We begin with a
decomposition lemma for acyclic locally finite Borel graphs (Lemma~\ref{path_decomp}). As an intermediate step towards
this lemma, we consider decompositions into certain types of trees that
themselves have suitable decompositions into paths. Recall that a
\define{tree} is a connected acyclic graph, a \define{leaf} of a tree is a
vertex of degree $1$, and a \define{splitting vertex} is a vertex of degree
at least $3$.
Say that a tree $T$ is \define{$n$-spindly} if there is at
most one leaf $l$ of $T$ so that for all distinct leaves $x, y$, if $l \notin
\{x,y\}$, then $d(x,y) > 2n$, and if $l \in \{x,y\}$, then $d(x,y) \geq n$.
The utility of spindly trees is the following lemma:

\begin{lemma}\label{spindly_lemma}
  Every finite $n$-spindly tree $T$ can be written as a union of
  edge-disjoint paths $p_0, p_1, \ldots$ each having length at least $n$,
  and which are end-ordered.
\end{lemma}
\begin{proof}
  We construct $p_0, p_1, \ldots, p_k$ by induction.
  Let $p_0$ be a path from one leaf to another leaf, having minimal length
  among such paths between leaves. Let the endpoints of $p_0$ be $x$ and $y$. We may
  assume $x \neq l$ for the distinguished leaf $l$ if it exists.

  For each
  vertex $z$ not in $p_0$, let $V_z$ be the set of $w$ such that there
  is a path $p$ from $z$ to $w$ such that no interior vertex of $p$
  is in $p_0$.
  Since $T$ is a tree, there is exactly one vertex in $T
  \restriction V_z$ which is contained in $p_0$. Let this vertex be $l_z$,
  which is a leaf in $T
  \restriction V_z$. For any leaf $w$ in $T \restriction V_z$, the
  distance $d(x,l_z) \leq d(w,l_z)$. Otherwise if $ d(w,l_z)< d(x,l_z)$, the path from $w$ to
  $y$ would have smaller length than $p_0$, but $p_0$ has minimal length.
  Hence, $d(w, l_z) \geq n$ since otherwise $d(x,l_z) \leq d(w,l_z) < n$
  which implies that $d(x,w) < 2n$ contradicting $T$ being $n$-spindly,
  since neither $x$ nor $w$ are equal to $l$. It follows that $T
  \restriction V_z$ is $n$-spindly witnessed by $l_z$.

  The lemma follows by inductively applying the lemma to all these
  $n$-spindly subgraphs of the form $T \restriction V_y$.
\end{proof}

\begin{remark}\label{n-spindly-decomp}
  Every locally finite $n$-spindly tree $T$ can be written as a union
  of edge-disjoint paths $p_0, p_1, \ldots$ that are of length at least $n$
  and which are end-ordered.
  That is,
  Lemma~\ref{spindly_lemma} remains true for infinite $n$-spindly trees.
  This is by an infinite iteration of the same process in the proof of
  Lemma~\ref{spindly_lemma} (or by a compactness
  argument).
\end{remark}

As an intermediate step towards our path decomposition, we prove a lemma
decomposing into $n$-spindly trees.

\begin{lemma}\label{spindly_decomp}
  Suppose $G$ is a locally finite acyclic graph on a Polish space $X$ of
  degree at least $2$, and
  $n \geq 1$. Then there are a $G$-invariant comeager Borel set $D$ and
  edge-disjoint Borel subgraphs $G_0, G_1, \ldots $ such that $\bigunion_i
  G_i = G \restriction D$, every connected component of $G_i$ is a finite
  $n$-spindly tree, and the sequence $G_0, G_1, \ldots$ is end-ordered.
\end{lemma}
\begin{proof}
  We give a construction in countably many steps. Let $d(i) = 3n6^{i}$. By
  \cite[Lemma 3.1]{MU16}, let $(A_i)_{i \in \N}$ be subsets of $X$ such
  that the elements of $A_i$ are pairwise of distance greater than $3
  d(i)$, and $D = \bigunion_i A_i$ is comeager and $G$-invariant. Before
  step $s$ we will have constructed edge-disjoint Borel subgraphs $G_0,
  G_1, \ldots, G_{s-1}$. Let $H_{i} = \bigunion_{j \leq i} G_j$. Let
  $H_{i,k}$ for $k \leq i$ be all the connected components $C$ in $H_i$ where $k$ is least
  such that $C$ is also a connected component of $H_k$. So $H_i$ is the
  disjoint union $H_i = \bigunion_{k \leq i} H_{i,k}$.

  Our induction hypotheses are as follows:

  \begin{enumerate}
  \item For every $i<s$ and $x \in A_i$, there is an edge incident to $x$ in
$H_i$.
  \item For every $k \leq s - 1$, the diameter of any connected component
  of $H_{s-1,k}$ is at most $d(k)$.

  \item For every $k \leq s - 1$, the distance between any two connected
  components of $H_{s-1,k}$ is at
  least $2 d(k)$.
  \item The distance between any two connected components of $H_{s-1}$
  is greater than $2n$.
  \end{enumerate}

  Note that these hypotheses imply that every edge in $G
  \restriction D$ will appear in some $G_i$. To see this suppose $x, x'$ are
  adjacent where $x \in A_s$ and $x' \in A_{s'}$. Then $x$ and $x'$
  must both be in connected components of $H_{\max(s,s')}$ by (1), and
  hence the same connected component by (4). Thus the edge $(x,x')$ must be
  in $H_{\max(s,s')}$ since $G$ is acyclic.

  Below we inductively define $G_s$, then prove each connected component of
  $G_s$ is $n$-spindly. Note that to satisfy part (1) of the induction
  hypothesis, we need to add an edge incident
  to each $x \in A_s$ to $G_s$ if there is not one already one in
  $H_{s-1}$. However, simply adding such edges by themselves may violate induction
  hypothesis (4). So we will need to inductively define $G_s$ to include
  paths to all nearby connected components of $H_{s-1,k}$ so they all become the
  same connected component in $H_s$. Hypotheses (2) and (3) give us control
  over this process so we can satisfy (4).

  To begin, let $G_{s,0}$ be the graph consisting of all vertices in
  $A_{s}$ (and no edges). Inductively, for $0 < i \leq s$, let $G_{s,i}
  \supset G_{s,i-1}$ be the union of $G_{s,i-1}$ with all paths of length
  at most $d(s-i)$ in the graph $G \setminus H_{s-1}$
  from vertices in $G_{s,i-1}$ to connected components of
  $H_{s-1,s-i}$.
  Since elements of $A_s$ have pairwise distance at least
  $3d(s)$ it is clear by induction that components of $G_{s,i}$ have
  diameter at most $2d(s-1) + \ldots + 2d(s - i)$, and hence components of
  $G_{s,s}$ have diameter at most $2d(s-1) + \ldots + 2 d(0)$.
  Similarly, the components of $G_{s,s}$ are pairwise of distance at least
  $3d(s) - 2d(s-1) - \ldots - 2 d(0)$.

  Let $A^0_s$ be the set of $x \in A_s$ that are not incident to any edge of
$H_{s-1}$ or $G_{s,s}$. (Hence, every $x \in A^0_s$ has $d(x,H_{s-1}) >
d(0) \geq 3n$). For each $x
\in A^0_s$, let $p(x)$ be the lex-least path of length $n$ in $G$ starting at
$x$. Let $A^1_s$ be the set of $x \in A_s$ that are leaves in $G_{s,s}$.
For $x \in A^1_s$, let
$p(x)$ be the lex-least path of length $n$ starting at $x$ in $G
\setminus (G_{s,s} \union H_{s-1})$. Such a path exists since every vertex
in $G$ has
degree at least $2$, and since if $y$ is a neighbor of $x$ that is not in
$G_{s,s}$, then there is no simple path of length at most $d(0) \geq 3n$ beginning $x,
y, \ldots$ that ends in an element of $H_{s-1}$ by the definition of
$G_{s,s}$.
Let $J_s = \{p(x) \colon x \in A^0_s \lor x \in A^1_s\}$ and let
$G_s = G_{s,s} \union J_s$. Clearly $H_s$ satisfies (1) by definition.

  Suppose $C$ is
  a connected component of $G_{s}$. We want to prove $C$ is $n$-spindly.
  Now $C$ contains a unique $x \in A_s$. We consider three cases. Case 1: if $x \in A^0_s$, then clearly $C$
  is just a path of length $n$, hence $C$ is $n$-spindly.
  Case 2: if $x \in A^1_s$,
  then let $p(x) = x, \ldots, z$ have endpoint $z$. In this case, $z$ is the
  distinguished leaf $C$; if $l$ is any other leaf of $C$, then $d(z,l)
  \geq n$ since $p(x)$ has length $n$.
  By the definition of $G_s$, any leaf in $C$ not equal to $z$ is the
  endpoint of a path from $G_{s,i-1}$ to $H_{s-1,s-i}$ for some $i$.
  Since any two connected components of
  $H_{s-1}$ have distance at least $2n$ by (4), all these leaves have
  distance pairwise greater than $2n$. So $C$ is $n$-spindly.
  Case 3: if $x \notin A^0_s$ and $x \notin A^1_s$, then
  all leaves of $C$ are endpoints of paths from $G_{s,i-1}$ to $H_{s-1}$ and have
  distance greater than $2n$, so $C$ is $n$-spindly.

  Now we verify parts (2) and (3) of the induction hypothesis. By construction
  of $G_s$, every connected component of $G_s$ has diameter at most $2d(s-1) +
  \ldots + 2d(0) + n \leq d(s) - 2d(s-1)$. Since connected components of
  $H_{s-1}$ have diameter at most $d(s-1)$ by our induction hypothesis,
  connected components of
  $H_{s,s}$ therefore have diameter at most $d(s)$. Similarly,
  the distance between any two connected
  components of $G_s$ is at least $3d(s) - 2d(s-1) - \ldots - 2d(0) - 2n
  \geq 2d(s) + 2d(s-1)$.
  Hence, connected components of $H_{s,s}$ have
  pairwise of distance at least $2d(s)$, since connected components of
  $H_{s-1}$ have diameter at most $d(s-1)$.
  Note that if $C$ is a connected component of $H_{s,k}$, then it is also a
  connected component of $H_{s',k}$ for all $s' <  s$. Hence, part (2) and
  (3) of the induction hypothesis are also true for all $k < s$.
  This verifies parts (2) and (3) of the induction hypothesis.


  Now we show that part (4) of the induction hypothesis holds. Suppose $C$
  is a connected component of $H_{s,s}$.
  We want to show that distance from $y \in C$ to any other connected component $C'$ of
  $H_{s,k}$
  is greater than $2n$ for $k \leq s$. When $k = s$, this follows from (3), so assume $k <
  s$.
  For a contradiction, let $y$ be a vertex in $C$ with $d(y,C') \leq 2n$.
  We may assume that $y \in G_{s}$, since if $y \in
  G_{s'}$ for $s' < s$, then $d(y,C')$ follows from our induction hypothesis.
  We may further assume $y \in G_{s,s}$. To see this, let $x \in C$ be
  the unique vertex in $C$ with $x \in A_s$. If $x \in A^0_s$, then $C =
  p(x)$, and $d(x,H_{s-1}) > 3n$, so $d(y,H_{s-1}) > 2n$ since $p(x)$ has
  length $n$. If $x \in A^1_s$, then any path of length at most $2n$ from
  $x \in p(z)$ to an element of $H_{s-1}$ must go through $y$ by our
  discussion after the definition of $p(x)$.

  So let $y \in G_{s,s}$ be so that $d(y,C') \leq 2n$. Let $y'$ be
  the closest element in $G_{s,s-k} \restriction C$ to $C'$. Hence by the
  the construction of $G_{s,s}$, we have
  $d(y,y') \leq d(k-1) + \ldots + d(0)$. Since $C'$ is distance at
  most $2n$ from $y$, $C'$ is distance at most $2n + d(k-1) + \ldots +
  d(0) < d(k)$ from $x'$. First suppose $x'$ is also a vertex in
  $G_{s,s-k-1}$. Then $x'$ would be an element of $G_{s,s-k-1}$ of distance $< d(k)$
  from an element of $H_{s,k}$, and so in the definition of $G_{s,s-k}$ there
  should have been a path added from $G_{s,s-k-1}$ to $C'$ in $G_{s,s-k}$.
  If $x'$ is not a vertex in $G_{s,s-k-1}$, then $x'$ must be part of a path added in
  $G_{s,s-k}$ from an element of
  $G_{s,s-k-1}$ to some connected component $C''$ of $H_{s-1,k}$. Since this path is of distance at most
  $d(k)$, this would imply that $C'$ and $C''$ are of distance $< 2d(k)$
  which contradicts part (3) of the induction hypothesis.
\end{proof}

We are now ready to prove our path decomposition lemma.

\begin{lemma}\label{path_decomp}
  Suppose $G$ is a locally finite acyclic graph on a Polish space and $n
  \geq 1$. Then
  there is a comeager Borel set $D$ such that $G \restriction D$ has a path
  decomposition of length at least $n$.
\end{lemma}
\begin{proof}
  We prove this lemma by combining Lemma~\ref{spindly_decomp} and
  Lemma~\ref{spindly_lemma} with the obvious derivative process to obtain
  sets of paths.

  Suppose $(D_i)_{i \in \N}$ is such that each $S \in D_i$ is a finite
  sequence of paths in $G$ that are end-ordered, and $(D_i)_{i \in \N}$ is
  end-ordered. Let $<_{(D_i)}$ (suppressing the indexing for
  clarity) be the partial order on the
  paths appearing in the elements of the $D_i$ where $p <_{(D_i)} p'$ if
  $p, p'$ share some vertex, and either $p, p'$ both appear in some
  sequence $S \in D_i$ where $p$ appears before $p'$, or $p$ is in an
  element of $D_i$ and $p'$ is in an element of $D_j$ for $i < j$.

  We begin by applying Lemma~\ref{spindly_decomp} to obtain a $n$-spindly
  decomposition $G_0, G_1, \ldots$ of $G$ restricted to some comeager
  $G$-invariant Borel set. If $C$ is an $n$-spindly connected component of
  some $G_i$, let $P(C)$ be the lexicographically least decomposition
  $(p_0, \ldots, p_k)$ of $C$ satisfying the conclusion of
  Lemma~\ref{spindly_lemma}. Letting $D_{i,0} = \{P(C) \colon \text{$C$ is
  a connected component of $G_i$}\}$, we obtain a sequence $(D_{i,0})_{i
  \in \N}$ of sets of finite sequences of paths in $G$, and the associated
  partial order $<_{(D_{i,0})}$ defined in the previous paragraph.

  Inductively, for $j \geq 0$, let $P_j$ be the set of $p$ appearing in some element of
  $D_{i,j}$ such that there is no $p' <_{(D_{i,j})} p$. Then let
  $D_{i,j+1}$ be the set of all sequence in $D_{i,j}$ with all elements of
  $P_j$ removed. These $P_j$ are our desired set of paths. Every path $p'$
  in each $S \in D_i$ must eventually appear in some $P_j$ since there are
  only finitely many $p$ such that $p <_{(D_{i,0})} p'$.
\end{proof}

A useful observation is that if $G$ is a graph with a path decomposition,
the decomposition may be assumed to consist of paths of bounded length.
This follows the fact that the intersection graph on paths has a countable
Borel coloring, and a derivative operation analogous to that of
Lemma~\ref{path_decomp}.

\begin{lemma}\label{good_path_decomp}
Suppose $G$ is a locally finite Borel graph with a Borel path decomposition $P_0,
P_1, \ldots$ of length at least $n$. Then $G$ admits a Borel path
decomposition $P_0', P_1', \ldots$ of length at least $n$ such that
every path $p \in P_i'$ has length at most $2n$.
\end{lemma}

\begin{proof}
  Every path of length greater than $2n$ can clearly be written as a finite
  union of paths of length between $n$ and $2n$. Hence, we may replace any
  path $p \in P_i$ of length greater than $2n$ by the lex-least finite set
  of paths of length between $n$ and $2n$ whose union is $p$. This gives a
  sequence $P_0, P_1, \ldots$ having every property of being a Borel path
  decomposition with the exception that the $P_i$ may not consist of vertex
  disjoint paths (but with the property that every path in every $P_i$ has
  length at most $2n$).

  Let
  $H$ be the graph on the paths $\bigunion_i P_i$ where distinct $p,p' \in
  \bigunion_i P_i$ are adjacent in $H$ if they share some vertex. Then $H$
  is a locally finite Borel graph and hence has a countable Borel coloring
  $c \from \bigunion_i P_i \to \N$ by \cite[Proposition 4.10]{KST}.

  Inductively, let $D_{i,0} = P_i$. For a fixed $j$, we can order the paths
  in $\bigunion_i D_{i,j}$ by $p
  <_{(D_{i,j})} p'$ if $p \in D_{i,j}$ and $p' \in D_{i',j}$ where
  either $i < i'$, or $i = i'$ and $c(p) < c(p')$. Now a construction
  identical to the last paragraph of the proof of Lemma~\ref{path_decomp}
  gives our desired Borel path decomposition.
\end{proof}

\begin{lemma}\label{realization_lemma}
  Suppose that $E$ is an abstract system of congruences on $n$ which is
  non-expanding, and $R = \{(S_1, T_1), \ldots,
  (S_k, T_k)\}$ is a minimal good generating set of $E$. Suppose also $a$ is a free Borel action of
  the group
  \[\Gamma = \langle \gamma_1 \ldots \gamma_k \mid
  \{\gamma_i^2 = 1 \colon T_i = \comp{S_i}\}\rangle \]
  on a Polish space $X$. If
  $G(a,\{\gamma_1, \ldots, \gamma_k\})$ has a Borel path decomposition of
  length at least $r$ for sufficiently large $r$ (depending on $E$), then
  there is an $a$-realization of $E$ with Borel pieces witnessed by
  \[\gamma_i \cdot \bigunion_{j \in S_i} A_j = \bigunion_{j \in
  T_i} A_j.\tag{*}\]
  Furthermore, if the space $X$ is assumed to be
  perfect, then the sets $A_1, \ldots, A_k$ can be chosen so each is
  nonmeager.
\end{lemma}
\begin{proof}
  Let $G$ be the graph $G = G(a,\{\gamma_1, \ldots, \gamma_k\})$.
  The idea of our proof is as follows. We first argue that there is a sufficiently
  large length $r$ so that given any path $p$ of length at least $r$ in $G$, if
  we have already assigned the endpoints of $p$ to be in elements of $A_0,
  \ldots, A_{n-1}$, then there is some way of consistently assigning the
  interior points of the path to elements of $A_0, \ldots, A_{n-1}$ so as
  to obey the congruences required in (*). Then we use a
  path decomposition of length at least $r$ for $G$ to inductively
  construct a realization of this
  system of congruences.

  Suppose that $g=g_l \ldots g_0$ is a reduced word in $\Gamma$, where $g_i
  \in \{\gamma_1^\pm, \ldots, \gamma_k^\pm\}$ are generators. If we begin at some $x \in
  X$, then such a reduced word of length $l+1$ gives a
  path of length $l+1$ in $G$: the path $x, g_0 \cdot x, \ldots, g_l \ldots
  g_0 \cdot x$.
  We give a definition concerning what elements of $A_0, \ldots, A_{n-1}$
  the elements of this path can belong to.
  Define functions $X$ and $Y$ on generators as follows:
  $X(\gamma_j) = S_j$, $Y(\gamma_j) = T_j$, $X(\gamma_j^{-1}) = T_j$, and $Y(\gamma_j^{-1}) = S_j$. Say that $n_0, \ldots, n_{l+1}$ is a
  \define{labeling of $g = g_l \ldots g_0$} if for all $i$, we have $n_i \in X(g_i)$ if and only if
  $n_{i+1} \in Y(g_i)$. So labelings correspond to acceptable assignments of
  the points $x, g_0 \cdot x, \ldots, g_l \cdots
  g_0 \cdot x$
  to the sets
  $A_0, \ldots, A_{n-1}$.

  We are interested in the ways labelings of $g$ may start and
  end.
  If $k,m \in n$, say a reduced word $g$ is \define{$(k,m)$-bad} if
  there is no labeling $n_0, \ldots, n_{l+1}$ of $g$ with $n_{0} = m$ and
  $n_{l+1} = k$.
  Say that $g$ is \define{bad} if there is some $k,m
  \in n$ such that $g$ is $(k,m)$-bad. We will use a pigeonhole principle
  argument to show there is a bound on the length of bad words.

  To begin, note that if $g = g_l \ldots g_0$ is bad, then $g_{l} \ldots g_{1}$ and $g_{l-1}
  \ldots g_0$ are also bad. That is, initial segments and final segments of
  bad words are bad.

  Suppose $g=g_l \ldots g_0$ is $(k,m)$-bad. Then exactly one of the following
  holds. Either
  \begin{enumerate}
  \item $m \in Y(g_l)$
  and $g$ is $(k,m')$-bad for every $m' \in Y(g_l)$, or
  \item $m \in \comp{Y(g_l)}$ and $g$ is $(k,m')$-bad for every $m' \in
  \comp{Y(g_l)}$.
  \end{enumerate}


  Fix a $(k,m)$-bad word $g$. Define a pair of associated sequences $V_{g,k}(i)$ and
  $W_{g,k}(i)$ where
  $(V_{g,k}(i),W_{g,k}(i)) = (\comp{X(g_i)},\comp{Y(g_i)})$ if $g_i \ldots
  g_0$ is $(k,m')$-bad for
  every $m' \in \comp{Y(g_i)}$ and $(V_{g,k}(i),W_{g,k}(i)) =
  (X(g_i),Y(g_i))$ otherwise. It is clear that there exist
  labelings $n_0, \ldots,
  n_{l+1}$ of $g$ where $n_i \in V_{g,k}(i)$ and
  $n_{i+1} \in W_{g,k}(i)$ for every $i$.
  Indeed, we have that $V_{g,k}(i) \E W_{g,k}(i)$ by definition, and $W_{g,k}(i) \subset
  V_{g,k}(i+1)$ for all $i \leq l$ or else $g$ is not a bad word.

  Suppose for a contradiction that there are infinitely many
  bad words. We break into two cases

  Case 1: suppose that there are arbitrarily long bad words $g$
  such that $g$ is $(k,m)$-bad for some $(k,m)$, and $W_{g,k}(i) = V_{g,k}(i+1)$ for all $i < l$. Hence $V_{g,k}(0) \E
  V_{g,k}(1) \E \ldots \E
  V_{g,k}(l)$. By the pigeonhole principle, and since initial
  segments and final segments of bad words are bad, we can find some
  bad
  word $g$ such that $g$ is $(k,m)$-bad, and 
  \[ g \text{ has length at least $2$ and } V_{g,k}(0) = W_{g,k}(l) \text{
  or } V_{g,k}(0) = \comp{W_{g,k}(l)} \tag{**}.\]
  We claim that this implies that either the word $g$ is not reduced, or
  the generating set of $E$ is not a minimal good generating set.

  First, we may assume that $g$ has minimal length among bad words
  with property (**), and so
  no proper subword of $g$ has property (**).

  If $V_{g,k}(0) = W_{g,k}(l)$, then the minimal length of $g$ among words with (**) implies that $g_0 \neq g_i^{\pm 1}$ for any $i > 0$. This implies that
  the generating set $R$ is not a minimal good generating set; the fact
  that $V_{g,k}(0) \E V_{g,m}(1)$
  follows from $V_{g,k}(1) \E W_{g,k}(1) = V_{g,k}(2) \E \ldots \E
  V_{g,k}(l) \E W_{g,k}(l)$ and $W_{g,k}(l) = V_{g,k}(0)$. In
  particular, removing the pair $(S_j,T_j)$ where $g_0 = \gamma_j$ would
  still generate $E$. Hence the generating set is not minimal.

  In the case that $V_{g,k}(0) = \comp{W_{g,k}(l)}$, we can also
  remove the pair $(S_j, T_j)$ where $g_0 = \gamma_j$, since there must be
  a generator witnessing $V_{g,k}(0) \E \comp{V_{g,k}(0)} = W_{g,k}(l)$ by our definition of a
  good generating set (see Definition~\ref{defn:good_generating}). In particular, a good generating set must contain
  every relation of the form $(S,\comp{S})$ where $S \E \comp{S}$.

  Case 2: suppose case $1$ does not hold. Then by the pigeonhole principle,
  and since initial segments and final segments of bad words are bad, we can
  find some $(k,m)$-bad word $g$ such that $V_{g,k}(0) = W_{g,k}(l+1)$, and
  $W_{g,k}(i) \subsetneq
  V_{g,k}(i+1)$ for some $i \leq l$. Then we can obtain a contradiction to the non-expansion of $E$
  by cyclically permuting the sequences to bring $V_{g,k}(i+1)$ to the $0$
  position and $W_{g,k}(i)$ to the $l$th position.

This finishes the proof that there are only finitely many bad words.

Now let $r$ be sufficiently large so that there are no bad words of length
$r$, and let $P_0, P_1, \ldots$ be a Borel path decomposition of $G$ of
length at least $r$. We may assume that this path decomposition satisfies
the conclusion of Lemma~\ref{good_path_decomp}. Now we inductively
construct a Borel $a$-realization $A_0, \ldots, A_{n-1}$ of $E$ in
countably many steps. After step $i$ we will have assigned each vertex
appearing in the paths in $P_j$ for $j \leq i$ to some $A_0, \ldots,
A_{n-1}$.

At step $i$ we will consider the paths $p \in P_i$. For each such path $p$, we
assign the vertices of $p$ to be the lex-least assignment to $A_0, \ldots,
A_{n-1}$ that is consistent with the requirement (*) in the statement of the
lemma. There is guaranteed to be such an assignment since we will have assigned
at most the start and end node of the path to $A_0, \ldots, A_{n-1}$ and since
the path has length at least $r$, the group element corresponding to it is not
bad.

At the end of this construction we will have assigned every element of $X$
to some $A_0, \ldots, A_{n-1}$. Since every edge in $G$ appears in some
path $p$, this ensures that the requirement (*) is satisfied at the end of
the construction.

To finish, we prove the ``furthermore'' statement at the end of the lemma.
Suppose that the space $X$ is perfect. We show that the sets $A_1, \ldots,
A_n$ can be chosen to be nonmeager. Notice that it suffices to have a path
decomposition where the first set $P_0$ of paths has a set of endpoints $D$
that is nonmeager. If this is then case, then we may may partition $D$ into
$k$ many nonmeager Borel sets since $X$ is perfect. Then we may assign
these $k$ sets to $A_1, \ldots, A_k$. This is because in our construction
above, the endpoints of the paths of $P_0$ may be assigned to $A_1, \ldots,
A_k$ arbitrarily.

  So we need to show that we can construct a path decomposition where the
  set of endpoints of paths in $P_0$ is nonmeager. To see this, observe
  that in our proof of Lemma~\ref{spindly_decomp} given the subsets
  $(A_i)_{i \in \N}$ of $X$ such that the elements of $A_i$ are pairwise of
  distance greater than $d(i)$, note that all the elements of the set $A_0$
  become endpoints of paths in $P_0$ in the final path decomposition.
  Hence, it suffices to show that $A_0$ can be chosen to be nonmeager in
  \cite[Lemma 3.1]{MU16}. To see this, note first that we can find a Borel
  nonmeager $k$-independent set. This is because $G^{\leq k}$ has a
  countable Borel coloring \cite[Proposition 4.10]{KST} and one of the
  color sets must therefore be a nonmeager $k$-independent Borel set $A_0$.
  Now apply \cite[Lemma 3.1]{MU16} to the graph $G \setminus A_0$ and the
  function $f(n) = d(n+1)$.
\end{proof}

We are now ready to prove Theorem~\ref{Baire_abstract}.
\begin{proof}[Proof of Theorem~\ref{Baire_abstract}.]
  We begin with the forward direction of Theorem~\ref{Baire_abstract}.
  Suppose $A_0, \ldots, A_{n-1}$ is a Baire measurable realization of an
  abstract system of congruence $E$ on $n$ where every $A_i$ is nonmeager.
  By Theorem~\ref{Wagon_abstract} it suffices to show that $E$ is non-expanding.
  For a contradiction, suppose there are sequences of sets $(V_i)_{i \leq k}$
and $(W_i)_{i \leq k}$ with $V_i, W_i \in \PP(m)$ such that $V_i \E W_i$ for
every $i \leq k$, $W_i \subset V_{i+1}$ for every $i < k$ and $V_0 \supsetneq
W_k$.  Let $A = \bigunion_{i \in V_0} A_{i}$ and $B = \bigunion_{i \in W_k}
A_i$. Let $\gamma$ be the product of the group elements witnessing $V_i \E W_i$
taken in increasing order for $i \leq k$. It follows that $\gamma \cdot A
\subseteq B$.  Clearly if $x \in A \setminus B$, then for all $n > 0$, $\gamma^n
\cdot x \notin A \setminus B$.

  Now there are two cases. First, if the rotation given by $\gamma$ is
  rational (i.e. periodic), this implies that $A \setminus B$ is not in any
  orbit of $\gamma$. This
  contradicts the fact that $A \setminus B$ is nonmeager.

  Second, suppose the rotation of $\gamma$ is aperiodic. Then $A \setminus
  B$ meets each orbit of $\gamma$ in at most one point which contradicts $A
  \setminus B$ being nonmeager as follows. If $A \setminus B$ was nonmeager,
  there would be an open set $U$ in which $A \setminus B$ is comeager. But
  since $\gamma$ is an irrational rotation, we can find some $n > 0$
  rendering $\gamma^n$ arbitrarily close to the identity, and hence some $n$
  for which $\gamma^n U \inters U \neq \emptyset$. Since $\gamma$ is a
  homeomorphism, this implies that both $A \setminus B$ and $\gamma^n \cdot
  (A \setminus B)$ are comeager in $\gamma^n U \inters U$. But then there
  is some $x$ so that $x \in A \setminus B$ and $\gamma^n \cdot x \in A
  \setminus B$ which is a contradiction. This finishes the proof of the
  forward implications.

  To prove the reverse implication, suppose that $E$ is non-complementing and non-expanding. Choose some $R =
  \{(S_1 T_1), \ldots, (S_k T_k)\}$ which minimally generates $E$, and let
  $\langle \gamma_1 \ldots \gamma_k \rangle$ be rotations of the $2$-sphere which
  generate a copy of $\F_k$.

  Now let $r$ be sufficiently large (so as to satisfy the hypothesis of
  Lemma~\ref{realization_lemma}). By Lemma~\ref{path_decomp}, we can find a
  comeager $G$-invariant Borel set $D$ so that there is a Borel path decomposition of length at least $r$ of
  $G \restriction D$. Let $a'$ be the restriction of the action of
  $\<\gamma_1, \ldots, \gamma_k\>$ to $D$.
  Then by Lemma~\ref{path_decomp} we can find a Borel
  $a'$-realization $A_0',  \ldots, A_{n-1}'$ of $E$. By
  the ``furthermore'' clause of Lemma~\ref{realization_lemma}, we can assume each of $A_0', \ldots,
  A_{n-1}'$ to be nonmeager.

  By Lemma~\ref{wagon_realization_lemma}, there is
  some realization $A_0'', \ldots, A_{n-1}''$  of $E$ on the $2$-sphere witnessed using (*). To finish
  our proof, replace $A_i''$ with $A_i'$ on $D$ to obtain a Baire measurable
  realization of $E$ on the $2$-sphere. That is, set $A_i = (A_i'' \inters
  \comp{D}) \union (A_i' \inters D)$.
\end{proof}

\section{Borel path decompositions from Borel end selections}

Suppose $f \from X \to X$. Say that $f$ is \define{aperiodic} if for all $x
\in X$ and $n \geq 1$, we have $f^n(x) \neq x$. Let $G_f$ be the graph
induced by $f$ where distinct $x_0, x_1 \in X$ are $G_f$-adjacent if
$f(x_0) = x_1$ or $f(x_1) = x_0$. Suppose $A \subset X$. Say that $A$ is
\define{forward recurrent} (with respect to $f$) if for every $x \in X$
there exists some $n \geq 0$ such that $f^n(x) \in A$.

We have the following lemma showing that bounded-to-one Borel functions
admit forward recurrent $r$-independent sets. Recall that a function $f
\from X \to Y$ is
\define{bounded-to-one} if there is some $k > 0$ such that for every $y \in
Y$, $|f^{-1}(y)| \leq k$.

\begin{lemma}\label{forward_rec_indep}
  Suppose $X$ is a standard Borel space and $f \from X \to X$
  is an aperiodic bounded-to-one Borel function. Then for every $r \geq 1$ there exists a
  Borel set $A \subset X$ that is forward recurrent and
  $r$-independent.
\end{lemma}
\begin{proof}
  Let $G_f^{\leq r}$ be the graph on $X$ where distinct $x, y \in X$ are
  $G_f^{\leq r}$-adjacent if $d(x,y) \leq r$. Since $G_f$ has bounded
  degree, $G_f^{\leq r}$ also has bounded degree. Hence, by \cite[Theorem
  4.6]{KST}, there is a Borel coloring $c$ of $G_f^{\leq r}$ with finitely
  many colors. Let $A$ be
  the set of $x \in X$ such that $c(x)$ is equal to the least number
  appearing infinitely often in the sequence $c(x), c(f(x)), c(f^2(x))
  \ldots$.
  Then for
  each $x$, all the elements of $A$ in the ($G$-)connected component of $x$ have
  the same color, and hence $A$ is $r$-independent, since $c$ is a coloring
  of $G_f^{\leq r}$. $A$ is forward recurrent by construction.
\end{proof}

Now we show that we can obtain Borel path decompositions from Borel end selections

\begin{lemma}\label{end_sel_path_decomp}
Suppose $G$ is an acyclic bounded degree Borel graph on $X$ such that there is
a Borel selection of finitely many ends in every connected component of
$G$. Then for every $n > 0$, $G$ admits a Borel path decomposition of
length at least $n$.
\end{lemma}
\begin{proof}
  We are given a bounded degree acyclic Borel graph $G$ on a standard Borel space
  $X$ where every vertex has degree at least $2$.
  First, by \cite[Theorem C]{HM} which builds on methods from \cite{Mi}, if there is a Borel function selecting finitely many
  ends from every connected component of $G$, then there is a Borel
  function selecting one or two ends in every connected component of $G$.
  Hence, we can partition $X$ into
  two $G$-invariant Borel sets $C_1, C_2$ so that
  $G \restriction C_1$ has a Borel selection of one end in each
  connected component, and $G \restriction C_2$ has a Borel selection of
  two ends in each connected component.

  Let $r(x)$ be the Borel function
  selecting one end in each connected component of $G \restriction C_1$. We
  may assume that $r(x)$ begins with the vertex $x$ (by either appending
  the path from $x$ to the start of the ray $r(x)$ if $x$ is not included
  in the ray, or deleting the vertices preceding $x$ if $x$ is included in
  the ray). Let
  $f(x)$ be vertex after $x$ in $r(x)$.
  Then it is easy to see that $f \from C_1 \to C_1$ generates the
  graph $G$.

  Let $B_2 \subset C_2$ be the Borel set of vertices
  vertices lying on the geodesic between the two ends chosen in $C_2$.
  Precisely, let $r_0(x)$, $r_1(x)$ be the functions selecting two ends in each
  connected component of $G \restriction C_2$. We may similarly assume that
  $r_0(x)$ and $r_1(x)$ begin with the vertex $x$, and let $f_0(x)$ be the
  vertex after $x$ in $r_0(x)$, and $f_1(x)$ be the vertex after $x$ in
  $r_1(x)$. Then $B_2 = \{x \in C_2 \colon f_0(x)
  \neq f_1(x)\}$. It is easy to see that
  every connected component of $G \restriction B_2$ is $2$-regular and
  every connected component of $G \restriction C_2$ contains exactly one
  connected component of $G \restriction B_2$.

  By Lemma~\ref{forward_rec_indep}, we can find a forward recurrent Borel set $A \subset
  C_1$ such that $A$ is $2n$-independent in $G$.
  Let $P_{0}^1$ be the set of lex-least paths of
  length $n$ which begin at some vertex of $A$, and let $B_1$ be the set of
  vertices contained in some element of $P_0^1$. If $x \in C_1 \setminus
  B_1$, let $[x]$ be the set of vertices $y$ for which there is a path $p$ from
  $x$ to $y$ for which no interior vertex of $p$ is in $B_1$. The forward
  recurrence of $A$ implies that for every $x \in C_1$, there is a unique
  forward-most element of $[x]$ under $f$. It is also clear that $G
  \restriction [x]$ satisfies the hypothesis of
  Remark~\ref{n-spindly-decomp}. For each $x$, the space of $n$-spindly
  decompositions is a compact space in the natural topology on all such
  decompositions.
  Hence, by compact uniformization \cite[Theorem 5.7.1]{Sr}, see also
  \cite[Theorem 18.18]{K}, there is a Borel way of selecting
  a unique a path decomposition of length at least $n$ for $G
  \restriction[x]$ for each $x \in C_1 \setminus B_1$. Hence, we can extend
  $P_0^1$ to a Borel path decomposition of length at least $n$ for $G
  \restriction C_1$.

  On $G \restriction C_2$, we can first partition $G \restriction
  B_2$ into a Borel set $P_{0}^2$ of finite paths of length at least $n$.
  if $x \in C_2 \setminus B_2$, let $[x]$ be the set of $y \in X$ such that
  there is a path $p$ from $x$ to $y$ for which no interior vertex of $p$ is
  in $B_2$. Once again, $G \restriction [x]$ is $n$-spindly.
  Hence by Remark~\ref{n-spindly-decomp} we can extend $P_0^2$ to a Borel path
  decomposition of length at least $n$ for $G \restriction C_2$.
\end{proof}

Using Adams end selection, we can use this lemma to show
that $\mu$-hyperfinite free actions of $\F_n$ have $\mu$-measurable
realizations of abstract systems of congruences that are non-complementing
and non-expanding.

\begin{thm}\label{mu_hyp_decomp}
  Suppose that $n \geq 2$, and $a$ is a free Borel action of $\F_n$ on a
  standard probability space $(X,\mu)$ that is $\mu$-hyperfinite. Then
  there is a $\mu$-measurable $a$-realization of every abstract system of
  congruences $E$ that is non-complementing and non-expanding.
\end{thm}
\begin{proof}
  Let $R$ minimally generate $E$. Pass to a free subgroup $\F_k \leq \F_n$
  where $k = |R|$. Let $S$ be the set of generators of $S$.
  By a theorem of Adams \cite[Lemma 3.21]{JKL}, on a conull set there is a
  Borel function selecting either one or two ends from each connected
  component of $G(a,S)$. Hence, the theorem follows from 
  Lemmas~\ref{end_sel_path_decomp} and \ref{realization_lemma}.
\end{proof}

When we apply Lemma~\ref{end_sel_path_decomp}, it will be useful to know
that end selections pass between finite index subgroups.

\begin{lemma}\label{finite_index}
  Suppose $a$ is a free Borel action of a finitely generated group $\Gamma$ on $X$. Let $\Delta \leq
  \Gamma$ be a finitely generated finite index subgroup of $\Gamma$, and $b$ be the restriction
  of the action of $a$ to $\Delta$. Then if $S \subset \Gamma$ and $R
  \subset \Delta$ are finite symmetric generating sets, then $G(a,S)$ has a
  Borel selection of finitely many ends if and only if $G(b,R)$ has a Borel
  selection of finitely many ends.
\end{lemma}
\begin{proof}
  Since $\Delta$ is finite index in $\Gamma$, each $G(a,S)$ connected
  component contains
  finitely many components of $G(b,R)$, and each connected component of
  $G(b,R)$ is bounded distance from every point in the connected component
  of $G(a,S)$ it is contained in. Hence, there is an effectively defined
  bijection between ends in a connected component of $G(a,S)$, and ends in
  each $G(b,R)$-component that it contains.

  More precisely, suppose $r = (x_i)_{i \in \N}$ is a ray representing an
  end in $G(a,S)$, and $C$ is a connected
  component of $G(b,R)$. We define a ray $f_C(r)$ in $G(b,R)
  \restriction C$ as follows. To each $x_i$ we associate the nearest point $y_i$
  in $C$, and let $f_C(r)$ be the lex least ray passing through all the
  points $(y_i)_{i \in \N}$, erasing loops. The map $f_C$ clearly lifts to a map sending a selection of
  finitely many ends in $G(a,S)$ to a selection of finitely many ends in
  $G(b,R)$. The reverse implication is similar.
\end{proof}

\section{Constructive realizations of non-expanding abstract systems of
congruences for $\PSL_2(\Z)$ acting on $\P^1(R)$}

The group $\PSL_2(\Z)$ acts on the space $\P^1(\R)$ of lines in $\R^2$
through the origin. By identifying such a line with the $x$-value $x \in \R
\union \{\infty\}$ of its intersection point with the line $y = 1$, it is
easy to see that this action is isomorphic to the action of $\PSL_2(\Z)$ on
$\R \union \{\infty\}$ by fractional linear transformations, where
$\begin{bmatrix} a & b \\ c & d \end{bmatrix}$ acts via $x \mapsto \frac{ax
+ b}{cx + d}$.

It is a standard fact (see \cite[VII.1]{Se}) that $\PSL_2(\Z)$ is generated
by the two transformations $\alpha(x) = x + 1$ and $\beta(x) = -1/x$, and moreover that it factors as the free product of $\langle \beta \rangle$ of order $2$ and $\langle \alpha\beta\rangle$ of order $3$.

The group $\PGL_2(\Z)$ is index 2 over $\PSL_2(\Z)$, and similarly is
generated by $\alpha(x) = x + 1$ and $\gamma(x) = 1/x$. Note that $\beta(x) =
\alpha^{-1}(\gamma(\alpha(\beta(\alpha^{-1}(x))))) = (-1 + 1/(1 + 1/(x -
1))) = -1/x$.

Let $\Irr$ denote the set of irrational numbers. Each $x \in \Irr$ has
a unique continued fraction expansion
\[x = a_0 + \frac{1}{a_1 + \frac{1}{a_2 + \frac{1}{a_3 + \cdots}}}\]
Where $a_0 \in \Z$ and $a_1, a_2, \ldots \in \Z^+$ are positive integers.
We note the continued fraction expansion of $x$ as $(a_0; a_1, \ldots)$.
The following lemma is standard.

\begin{lemma}\label{tail_equiv}
Let $f \from \Irr \to \Irr$ be the function given by
\[f(x) = \begin{cases} x - 1 & \text{ if $x > 0$} \\
1/x & \text{ if $x \in (0,1)$}\\
x + 1 & \text{ if $x < 0$}
\end{cases}\]
Then $f$ generates the orbit equivalence relation of $\PGL_2(\Z)$ on
$\Irr$, and so $x, y \in \Irr$ are in the same orbit if and only if their
continued fraction expansions are tail equivalent.
\end{lemma}
\begin{proof}
  The equivalence relation generated by $f$ is clearly contained in the orbit
  equivalence relation of $\PGL_2(\Z)$, since $f$ is defined piecewise by fractional linear
  transformations.

  Recall that two continued fraction expansions $(a_0; a_1,
\ldots)$ and $(b_0; b_1, \ldots)$ are tail equivalent if
there exists some $n,m > 0$ such that $a_{n + i} = b_{m + i}$ for all $i
\geq 0$. Since
\[f\left(a_0 + \frac{1}{a_1 + \frac{1}{a_2 + \frac{1}{a_3 + \cdots}}} \right)
= \begin{cases}
(a_0 - 1) + \frac{1}{a_1 + \frac{1}{a_2 + \frac{1}{a_3 + \cdots}}} & \text{ if $a_0 > 0$}\\
a_1 + \frac{1}{a_2 + \frac{1}{a_3 + \cdots}} & \text{ if $a_0 = 0$}\\
(a_0 + 1) + \frac{1}{a_1 + \frac{1}{a_2 + \frac{1}{a_3 + \cdots}}} & \text{ if $a_0 < 0$}
\end{cases}
\]
it is clear that if $x$ and $y$ are tail equivalent, then there are in the
same equivalence class of the equivalence relation generated by $f$.

To finish, since $\alpha(x) = x+1$ and $\gamma(x) = 1/x$ generate
$\PGL_2(\Z)$, it suffices to show that if $x \in \Irr$, then $x$, $x+1$, and
$1/x$ are tail equivalent. It is trivial to see that $x$ and $x + 1$ are
tail equivalent. That $x$ and $1/x$ are tail equivalent is clear when $x >
0$. When $x < 0$, since one of $x$ and $1/x$ are less than $-1$, by
swapping $x$ and $1/x$, we may
assume the continued fraction expansion of $x$ is $x = a + \frac{1}{b +
C}$, where $a \leq -2$, and $b \geq 1$. Then apply the following identity:
\[\frac{1}{a + \frac{1}{b + C}} = -1 + \frac{1}{1 + \frac{1}{(-a-2) +
\frac{1}{1 + \frac{1}{(b-1) + C}}}}.\]
Note that $-a-2 \geq 0$ and $b - 1 \geq 0$. If either if these two terms
are equal to zero, this just removes the corresponding term in the continued fraction
expansion, since $\frac{1}{0 + \frac{1}{a_n + C}} = a_n + C$.
\end{proof}

\begin{cor}\label{end_sel_psl2}
  Let $a$ be the restriction of the action of $\PSL_2(\Z)$ to the
  irrationals. Let $S = \{\alpha, \beta\}$ be the set of generators $\alpha(x) = x+1$ and
  $\beta(x) = -1/x$. 
  Then there is a Borel
  selection of one end in each equivalence class of $G(a,S)$.
\end{cor}
\begin{proof}
By Lemma~\ref{tail_equiv}, there is a Borel selection of one end in the graph
$G(a',\{\alpha,\gamma\})$, where $a'$ is the action of $\PGL_2(\Z)$ on
$\Irr$, and $\gamma(x) = 1/x$. Hence, this corollary follows by
Lemma~\ref{finite_index}, since $\PGL_2(\Z)$ is index 2 over $\PSL_2(\Z)$.
\end{proof}

The action of $\PSL_2(\Z)$ is free modulo a countable set, since if
$x = (ax + b)/(cx + d)$, then $x$ is the solution to a quadratic equation with
integer coefficients.
To finish, we need to analyze the
countable set on which the action is nonfree.

\begin{lemma}\label{comm_stab}
For every $x \in P^1(\mathbb{R})$, the stabilizer $\Stab(x)$ of $x$ in
$\PSL_2(\mathbb{Z})$ is cyclic.
\end{lemma}
\begin{proof}
  It suffices to show for all $x$ that $\Stab(x)$ is a solvable subgroup of $\PSL_2(\mathbb{Z})$
  containing no involution.  Indeed, as $\PSL_2(\mathbb{Z}) \cong (\Z/2\Z)
  * (\Z/3\Z)$, it follows from the Kurosh subgroup theorem 
  \cite[Theorem 7.8]{C} that
  all solvable subgroups are either cyclic or the free product of two involutions, and we are done upon precluding
  the latter alternative.
  
  Towards that end, first observe that the action of $\PGL_2(\mathbb{R})$ on $P^1(\mathbb{R})$ is transitive,
  and thus all stabilizers are conjugate to the stabilizer of the point at infinity.  This stabilizer is isomorphic to
  the group of affine transformations of the real line, and in particular is solvable.  Returning to $\PSL_2(\mathbb{Z})$,
  it follows that the stabilizer of every point is a subgroup of a solvable group, and hence is itself solvable.
  
  It remains to show that every nontrivial involution in $\PGL_2(\mathbb{Z})$ acts freely on $P^1(\mathbb{R})$.  But this
  is immediate as all such involutions are conjugate to $\beta \from x \mapsto -1/x$, which has no fixed point.
\end{proof}

We can now prove Theorem~\ref{PSL2_thm} from the introduction.

\begin{proof}[Proof of Theorem~\ref{PSL2_thm}.]
Let $R$ be a minimal relation generating $E$. Let $k = |R|$.
There is a finite index copy of $\F_2$ in $\PSL_2(\Z)$ and hence a
finite index copy of $\F_k$. Let the free generating set of $\F_k$ be $S$. Let $a$ be the restriction of the action of
$\PSL_2(\Z)$ to this copy of $\F_k$. Let $F \subset X$ be the subset on
which the action of $F$ is free. By Lemma~\ref{comm_stab}, the action
of $\PSL_2(\Z)$ on $F$ has cyclic stabilizers, and so by Lemma~\ref{wagon_realization_lemma},
there is a realization of $E$ witnessed by letting the generators $S$ of
$F_k$ witness the elements of $R$. Since $X \setminus F$ is a subset of
the quadratic rationals it is countable, and so the sets
realizing $E$ on $X \setminus F$ are Borel.

Now on $F$, the graph $G(a \restriction F,S)$ has a Borel selection of
finitely many ends by Corollary~\ref{end_sel_psl2} and Lemma~\ref{finite_index}.
Hence, by Lemma~\ref{end_sel_path_decomp} we have a Borel path
decomposition and hence by Lemma~\ref{realization_lemma} there is a
realization of $E$ on $a \restriction F$ once again with the $i$th
generator witnesses the $i$th congruence in $R$. The theorem follows by
taking the union of these two realizations.
\end{proof}

\section{Applications of path decompositions in Borel combinatorics}
\label{comb_applications}

If $G$ is a locally finite acyclic Borel graph, then path decompositions
for $G$ give a very strong type of unfriendly coloring:

\begin{lemma}\label{path_unfriendly}
  Suppose $G$ is a locally finite acyclic Borel graph on $X$ where every
  vertex has degree at least $2$. Then if $G$ has a Borel path decomposition
  of length at least $4$, then $G$ admits a Borel unfriendly coloring.
  Indeed, there is a Borel function $c \from X \to 2$ such that for every
  $x$, $|\{y \in N(x) \colon c(x) = c(y)\}| \leq 1$.
\end{lemma}
\begin{proof}
  Suppose $P_0, P_1, \ldots$ is the Borel path decomposition of $G$ of
  length at least $4$. We may assume that this path decomposition satisfies
  the conclusion of Lemma~\ref{good_path_decomp}.

  We inductively construct $c$. At step $i$ we will ensure that every
  vertex in a path $p \in P_i$ has been colored. For all such paths $p \in
  P_{i}$, inductively, the only vertices in $p$ that can have already been
  colored must be endpoints of $p$. Hence, there is some extension of our
  partial coloring so that every vertex of $p$ has at most one adjacent vertex of
  the same color, and the endpoint of $p$ have neighbors of the opposite
  color. For example, alternate between the two colors along $p$, possibly
  breaking parity once in the middle of the path. (The reason here paths of
  length $3$ cannot work is that if the endpoints of such a path were already assigned opposite
  colors, one of the endpoints would then gain another vertex of the same
  color).
  Since $P_i$ is a path
  decomposition, each vertex is an interior vertex of at most one path, and
  every edge is contained in some path. Hence, our final coloring $c$ of
  $X$ has the desired property that each vertex has at most one neighbor of
  the same color.
\end{proof}

By combining this Lemma with Lemma~\ref{path_decomp}, we obtain
Theorem~\ref{unfriendly_coloring} as a Corollary.

Suppose $G$ is an acyclic locally finite Borel graph where every vertex has
degree at least $3$. Then an almost identical greedy construction shows that if $G$ has a path
decomposition of length at least $3$, then $G$ has a Borel perfect
matching, and if $G$ has maximum degree $d$, then $G$ has a Borel
$d$-list-coloring for any Borel assignment of lists to edges of $G$. For example, this gives a new way of proving a Baire
measurable version of Vizing's theorem for acyclic bounded degree Borel
graphs, and the existence of Baire measurable perfect matchings for acyclic
locally finite Borel graphs.


\begin{thebibliography}{99}

\bibitem[C]{C} D. E. Cohen, Combinatorial Group Theory: a topological
approach, {\it Cambridge
University Press}, (1989).

\bibitem[CK]{CK} C.T. Conley and A.S. Kechris, Measurable chromatic and independence numbers for ergodic graphs and group actions, {\it Groups Geom. Dyn}., {\bf 7} (2013), 127--180.

\bibitem[CS]{CS} T. Cie\'sla and M. Sabok, Measurable Hall's theorem for
actions of abelian groups, \texttt{https://arxiv.org/abs/1903.02987}.

\bibitem[D]{D} R. Diestel, {\it Graph Theory}, Fourth, Graduate Texts in
Mathematics, vol. 173, Springer, Heidelberg, (2010).

\bibitem[DF]{DF} R. Dougherty and M. Foreman, Banach-Tarski decompositions using sets with the property
of Baire. {\it J. Amer. Math. Soc.}, {\bf 7(1)} (1994), 75-124.

\bibitem[DHK]{DHK} L. Dubins, M. W. Hirsch, and J. Karush, {\it Scissor
congruence}, Israel J. Math. \textbf{1} (1963), 239--247.

\bibitem[G]{G} R. J. Gardner, {\it Convex bodies equidecomposable by
locally discrete groups of isometries}, Mathematika \textbf{32} (1985),
1--9.

\bibitem[GJ]{GJ} S. Gao and S. Jackson, {\it Countable abelian group
actions and hyperfinite equivalence relations}. Invent. Math. \textbf{201}
(2015), no. 1, 309--383.

\bibitem[GJKS]{GJKS} S. Gao, S. Jackson, E. Krohne, and B. Seward {\it
Forcing constructions and countable Borel equivalence relations}. Preprint.
2015. 

\bibitem[GMP16]{GMP16} L. Grabowski, A. M\'ath\'e, O. Pikhurko,
{\it Measurable equidecompositions for group actions with an expansion
property}, \texttt{https://arxiv.org/abs/1601.02958}.

\bibitem[GMP17]{GMP17} L. Grabowski, A. M\'ath\'e, O. Pikhurko, {\it Measurable
circle squaring}, Ann. of Math 185 (2017), 671--710.

\bibitem[HM]{HM} G. Hjorth and B.D. Miller, {\it Ends of graphed equivalence
relations, II}, Israel Journal of Mathematics, \textbf{169} (2009), 393-415.

\bibitem[JKL]{JKL} S. Jackson, A.S. Kechris, and A. Louveau, {\it countable
Borel equivalence relations}, J. Math. Logic, \textbf{2} (2002), 1--80.


\bibitem[Ka]{Ka} S. Katok, {\it Fuchsian groups}, Chicago Lectures in
Mathematics, 1992.

\bibitem[K]{K} A.S. Kechris, {\it Classical descriptive set theory},
Springer, 1995.


\bibitem[KM]{KM} A.S. Kechris and A.S. Marks, {\it Descriptive graph
combinatorics}. Preprint, 2016. 

\bibitem[KST]{KST} A.S. Kechris, S. Solecki and S. Todorcevic, {\it Borel
chromatic numbers}, Adv. Math., \textbf{141} (1999), 1--44.

\bibitem[L88]{L88} M. Laczkovich, {\it Closed sets without measurable
matching}, P. Amer. Math. Soc. \textbf{103} (1988), no. 3, 894--896.

\bibitem[L90]{L90} M. Laczkovich, {\it Equidecomposability and discrepancy; a
solution of Tarski's circle-squaring problem}, J. Reine Angew. Math.
\textbf{404} (1990), 77-117.

\bibitem[L92]{L92} M. Laczkovich, {\it Decomposition of sets with small
boundary}, J. Lond. Math. Soc. \textbf{46} (1992), 58--64.

\bibitem[LN]{LN} R. Lyons and F. Nazarov, Perfect matchings as IID factors of non-amenable groups,
{\it European J. Combin.}, {\bf 32} (2011), 1115--1125.

\bibitem[MU16]{MU16} A. Marks and S. Unger, Baire measurable paradoxical
decompositions via matchings, {\it Advances in Math.} \textbf{289} (2016),
397--410.

\bibitem[MU17]{MU17} A. Marks and S. Unger, Borel circle squaring, {\it
Ann. of Math.} 186 (2017), 581--605.

\bibitem[Ma]{Ma} A. M\'ath\'e, {\it Measurable Equidecompositions}
Proc. Int. Cong. of Math. 2018, Rio de Janeiro, \textbf{2} 1709--1728.

\bibitem[Mi]{Mi} B.D. Miller, {\it Ends of graphed equivalence relations, I},
Israel Journal of Mathematics, \textbf{169(1)} (2009), 375--392.

\bibitem[N]{N} M. Nadkarni, {\it On the existence of a finite invariant
measure}. Proc. Indian Acad. Sci. Math. Sci., \textbf{100}, (1990) 203--220.

\bibitem[SS]{SS} S. Schneider and B. Seward, {\it Locally nilpotent groups and
hyperfinite equivalence relations}, Preprint.

\bibitem[Sr]{Sr} S.M. Srivastava, {\it A {C}ourse on {B}orel {S}ets},
Springer-Verlag, New York, (1998).

\bibitem[Se]{Se} J.P. Serre, {\it A {C}ourse in {A}rithmetic}. Graduate
Texts in Mathematics, vol. 7. Springer-Verlag, New York, (1973).

\bibitem[T]{T} A. Tarski, Probl\'eme 38, Fund. Math. \textbf{7} (1925), 381.

\bibitem[W]{W} S. Wagon, {\it The Banach Tarski Paradox}, University Press,
Cambridge, 1986.

\bibitem[Weh]{Weh} F. Wehrung, Baire paradoxical decompositions need at
least six pieces, {\it Proc. Amer. Math. Soc.}, {\bf 121(2)} (1994),
643--644.

\end{thebibliography}
\end{document}